%% file: ABR01.tex
\documentclass[a4paper,12pt,leqno,english]{article}
\usepackage[utf8]{inputenc}
\usepackage[T1]{fontenc}
\usepackage{babel}
\usepackage{amsmath}
\usepackage{amssymb}
\usepackage{enumitem}
\usepackage{graphicx}
\usepackage{hyperref}
\usepackage{tcolorbox}
\usepackage{pictex}

\input rs_macros.tex

\numberwithin{equation}{section}

\newtheorem{theorem+}           {Theorem}      [section]
\newtheorem{definition+}  [theorem+]  {Definition}
\newtheorem{lemma+}  [theorem+]  {Lemma}
\newtheorem{corollary+}  [theorem+]  {Corollary}
\newtheorem{proposition+}  [theorem+]  {Proposition}
\newtheorem{example+}  [theorem+]  {Example}
\newtheorem{problem+}  [theorem+]  {Problem}
\newtheorem{remark+}  [theorem+]  {Remark}

\newenvironment{theorem}{\begin{theorem+}\sl}{\end{theorem+}\rm}

\newenvironment{lemma}{\begin{lemma+}\sl}{\end{lemma+}\rm}
\newenvironment{corollary}{\begin{corollary+}\sl}{\end{corollary+}\rm}
\newenvironment{proposition}{\begin{proposition+}\sl}{\end{proposition+}\rm}

\newenvironment{proof}{\medbreak\noindent{\bf  Proof:}\rm}{\hfill$\square$\rm}
\newenvironment{prooftx}[1]{\medbreak\noindent{\bf 
    #1:}\rm}{\hfill$\square$\rm}

\title{{\Large \bf 
Polynomials with exponents in  compact convex sets  and 
associated weighted extremal functions: \\
The Siciak-Zakharyuta theorem}}
\author{{Benedikt Steinar Magnússon, Álfheiður Edda Sigurðardóttir} \\
{and Ragnar Sigurðsson}}

\date{{}} 
\begin{document}
\maketitle

\begin{abstract} \noindent
The classical Siciak-Zakharyuta theorem states that the Siciak-Zakharyuta function $V_{E}$ of a subset $E$ of $\mathbb C^n$, also called a pluricomplex Green function or global exremal function of $E$, equals the logarithm of the Siciak function $\Phi_E$ if $E$ is compact. The Siciak-Zakharyuta function is defined as the upper envelope of functions in the Lelong class that are negative on $E$, and the Siciak function is the upper envelope of $m$-th roots of polynomials $p$ in $\mathcal{P}_m(\mathbb C^n)$ of degree $\leq m$ such that $|p|\leq 1$ on $E$. We generalize the Siciak-Zakharyuta theorem to the case where the polynomial space ${\mathcal P}_m(\mathbb C^n)$ is replaced by ${\mathcal P}_m^S(\mathbb C^n)$ consisting of all polynomials with exponents restricted to sets $mS$, where $S$ is a compact convex subset of $\mathbb R^n_+$ with $0\in S$. It states that if $q$ is an admissible weight on a closed set $E$ in $\mathbb C^n$ then $V^S_{E,q}=\log\Phi^S_{E,q}$ on $\mathbb C^{*n}$ if and only if the rational points in $S$ form a dense subset of $S$.

\end{abstract}

\section{Introduction}
\label{sec:01}

\noindent A polynomial in $n$ complex variables 
$p(z)=\sum_{\alpha\in  \N^n} a_\alpha z^\alpha$, $z\in \C^n$ is supported in a set $S\subset \R^n_+$ if $a_\alpha\neq 0$ implies that $\alpha \in S$. Given a compact $S\subset \R^n_+$ and $m\in \N$, let the space ${\mathcal P}^S_m(\C^n)$ consist  of  all polynomials 
$p$ supported in $mS$, of the form $$p(z)=\sum_{\alpha\in (mS)\cap \N^n} a_\alpha z^\alpha, \qquad z\in \C^n.$$
If $S$ is convex and $0\in S$ then ${\mathcal P}^S(\C^n)=\cup_{m\in \N}{\mathcal P}^S_m(\C^n)$ forms a graded polynomial ring where the degree of a polynomial $p$ is the minimal $m$ such that $p$ is supported in $mS$.  The standard grading is obtained when $S$ is the unit simplex $\Sigma$, the convex hull of the union of $0$ and the unit basis $\{e_1,\dots,e_n\}$. Polynomials with prescribed supports are often called sparse polynomials.

Bayraktar \cite{Bay:2017} introduces methods from pluripotential theory to find the zero distribution of systems of random sparse polynomials as their degree tends to infinty. That paper is the naissance of pluripotential theory with respect to convex bodies, further developed by Bayraktar, Bloom, Levenberg and their collaborators in \cite{BayBloLev:2018, BayBloLevLu:2019, BayHusLevPer:2020}. A self-contained exposition of the background to  this paper is found in Magnússon, Sigurðar\-dóttir, Sigurðsson and Snorrason \cite{MagSigSigSno:2023}.

A fundamental result from classical pluripotential theory is the Siciak-Zakharyuta theorem which states an equivalence between two extremal functions, 
$V_{E,q}=\log \Phi_{E,q}$, 
for every admissible weight $q:E\to\R\cup\{+\infty\}$ on a compact subset $E$ of $\C^n$, see Siciak \cite{Sic:1981}.
Here $V_{E,q}(z)$ is the pluricomplex Green function, defined as the supremum of functions $u\in \mathcal{PSH}(\C^n)$ satisfying a growth estimate $u\leq \log^+|\cdot|+c_u$ for some constant $c_u$ and such that $u|_E\leq q$. The Siciak function $\Phi_{E,q}$ is the supremum of $|p|^{1/m}$ where $p$ is an $m$-th degree polynomial satisfying $\|pe^{-mq}\|_E\leq 1$, with $\|\cdot\|_E$ denoting the $\sup$-norm on $E$. 
For a proof of the original Siciak-Zakharyuta theorem  with $q=0$ and a reference on plurpotential theory, see Klimek \cite{Kli:1991} and for a reference on the weighted theory see Bloom's Appendix B in \cite{SaffTotik:1997}.

\smallskip
The definition of $\Phi_{E,q}$ can be modified using the grading of the ring $\mathcal P^S(\C^n)$. Instead of regarding $|p|^{1/m}$ where $m$ is the standard degree of $p$, we take $m$ such that $p$ is supported in $mS$. For every function $q\colon E\to \R\cup\{+\infty\}$
defined on  $E\subseteq\C^n$  and
$m\in \N^*=\{1,2,3,\dots\}$ we define  Siciak functions with respect to $S$, $E$, $q$ and $m$ by
$$
\Phi^S_{E,q,m}(z)=\sup\{|p(z)|^{1/m}\,;\, p\in {\mathcal P}^S_m(\C^n),
\|pe^{-mq}\|_E\leq 1\}, \qquad z \in \C^n,
$$   
and the Siciak function with respect to $S$, $E$ and $q$ by 
$$\Phi^S_{E,q}=\varlimsup_{m\to\infty}\Phi^S_{E,q,m}.$$ 

To generalize the Siciak-Zakharyuta theorem we study an extremal function $V^S_{E,q}$ to match $\log \Phi^S_{E,q}$. To do so, the logarithmic growth of $V_{E,q}$ is replaced by growth that reflects the grading of $\mathcal{P}^S(\C^n)$. The supporting function $\varphi_S$ of $S$, $\varphi_S(\xi)=\sup_{s\in S}\scalar s\xi$ with $\xi\in \R^n,
$ is positively homo\-geneous and convex and determines the
set $S$ uniquely. The logarithmic
supporting function
$H_S\in \PSH(\C^n)$ of $S$ is defined by 
$$H_S(z)=\varphi_S(\log|z_1|,\dots,\log|z_n|)$$
on $\C^{*n}=(\C^*)^n$ and the definition is extended to the coordinate 
hyperplanes  by
$$
H_S(z)=\varlimsup_{\C^{*n}\ni w\to z}H_S(w), \qquad z\in \C^n\setminus
\C^{*n}.
$$
The class $\L^S(\C^n)$ 
is defined as the set of all $u\in \PSH(\C^n)$ satisfying a growth
estimate of the form $u\leq H_S+c_u$ for some constant $c_u$. 
The Siciak-Zakharyuta 
function  with respect to $S$, $E$ and $q$ is defined by 
$$
V^S_{E,q}(z)=\sup\{u(z)\,;\, u\in \L^S(\C^n),\, u|_E\leq q\},
\qquad z \in \C^n.
$$ 

A function $p\in \OO(\C^n)$ is in  ${\mathcal
  P}^S_m(\C^n)$ if and only if $\log|p|^{1/m}$ is in $\L^S(\C^n)$, by \cite{MagSigSigSno:2023}, 
  Theorem 3.6.  By \cite{MagSigSigSno:2023},  Proposition 2.2, $\log\Phi^S_{K,q}$ is the supremum of the subclass of all $u\in\L^S(\C^n)$ with $u|_K\leq q$ that are of the form $u=\log|p|^{1/m}$ for some $p\in \mathcal P^S_m(\C^n)$. This implies that $\log \Phi^S_{E,q}\leq V^S_{E,q}$. 

\smallskip

The standard simplex
$\Sigma$, 
which yields 
the standard grading of polynomials, 
has the supporting function 
$\varphi_\Sigma(\xi)=\max\{\xi_1^+\dots,\xi_n^+\}$
where $\xi_j^+=\max\{\xi_j,0\}$.  Its 
loga\-rithmic supporting function is 
$H_\Sigma(z)=\log^+\|z\|_\infty$, which implies that
the class   
$\L^\Sigma(\C^n)$ is the standard Lelong class $\L(\C^n)$. 
We drop the superscript $S$ in the case $S=\Sigma$ and
the subscript $q$ in the case $q=0$.

A function $q\colon E\to \R\cup\{+\infty\}$ 
is an
{\it admissible weight  with respect to $S$
on $E$} if 
\begin{enumerate}
    \item[\textbf{(i)}] $q$ is lower semi-continuous,
    \item[\textbf{(ii)}] the set  $\{z\in E \,;\, q(z)<+\infty\}$ is non-pluripolar, and
    \item[\textbf{(iii)}]  if $E$ is unbounded, $\lim_{z\in E, |z|\to \infty}(H_S(z)-q(z))=-\infty$.
\end{enumerate}
Condition \textbf{(i)} 
excludes any function with a negative pole, which otherwise might force $V^S_{E,q}$ to be the supremum over an empty family. Without condition \textbf{(ii)} there would exist an $u\in \mathcal{PSH}(\C^n)$ such that $u=-\infty$ on $\{z\in E \,;\, q(z)<+\infty\}$. Then $u+c\leq q$ on $E$ for every constant $c$ so if $u\in \L^S$ then $V^{S}_{E,q}(z)=+\infty$ for all $z\in \C^n$ such that $u(z)>-\infty$. Admissible weights ensure that the upper regularization 
$V^{S*}_{E,q}$ is a member of $\L^S(\C^n)$, see \cite{MagSigSigSno:2023}, Propositions 4.5 and 4.6. The condition \textbf{(iii)} additionally implies that the Siciak-Zakharyuta function satisfies a growth estimate $V^{S}_{E,q}\geq H_S-c$ for some constant $c$, which is used in \cite{MagSigSigSno:2023},  Proposition 4.6, to show that $V^S_{E,q}=V^S_{K,q}$ and
$\Phi^S_{E,q}=\Phi^S_{K,q}$ for some compact $K\subseteq E$. Therefore admissible weight allow us to study mostly $E=K$ compact.

\medskip
Exponents in $mS\cap \N^n$ can only appear as dilates of rational points from $S$ so ${\mathcal P}^S_m(\C^n)={\mathcal P}^{S'}_m(\C^n)$ 
for $S'=\overline{S\cap \Q^n}$.  If  $S\cap \Q^n$ does not form a dense subset of $S$, then there exists
a point $\xi\in \R^n$ for which $\varphi_{S'}(\xi)<\varphi_S(\xi)$.
For every admissible weight $q$ on a compact set $K$
the function  $z\mapsto H_S(z)-c$, where 
$c=\max_{w\in K}H_S(w)-\min_{w\in K}q(w)$,
is in $\L^S(\C^n)$ and $\leq q$ on $K$.  Hence $H_S-c\leq
V^S_{K,q}$.  We also have $V^{S'}_{K,q}\leq V^{S'*}_{K,q}\leq H_{S'}+c'$ for 
some constant $c'$. 
By homogeneity of the supporting functions
we have for $r>0$ sufficiently large  that 
$\varphi_{S'}(r\xi)+c' < \varphi_S(r\xi)-c$, 
so if we set $z=(e^{r\xi_1},\dots,e^{r\xi_n})$, then
$$
\log\Phi^{S}_{K,q}(z)=\log\Phi^{S'}_{K,q}(z)
\leq V^{S'}_{K,q}(z)\leq H_{S'}(z)+c'
< H_{S}(z)-c \leq V^S_{K,q}(z).
$$
To summarize, we have observed that if $S\cap\Q^n$ is not dense in
$S$, then for every ad\-miss\-ible weight $q$ on a compact set $K$ we have
$V^S_{K,q}\neq \log\Phi^{S}_{K,q}$. This necessary condition on $S$ is almost sufficient, as seen by our main result.

\begin{theorem}
  \label{thm:1.1} 
Let $S\subset \R^n_+$ be compact and convex with $0\in S$ and let 
$q$ be an admissible weight on a closed $E\subset \C^n$.   Then 
$V^S_{E,q}(z)=\log\Phi^S_{E,q}(z)$ for every $z\in \C^{*n}$ if and only if $S\cap \Q^n$ is dense in $S$. Furthermore, if $V^S_{E,q}$ is lower semi-continuous then the equality extends to all $z\in \C^n$.
\end{theorem}

The last statement follows from the fact that $\Phi^S_{K,q}$ is lower semi-continuous, as it is the supremum over a family of continuous functions. The only remaining problem is to characterize when $V^S_{E,q}$ is lower-semicontinuous. The supremum over a family of continuous plurisubharmonic functions is a lower-semicontinuous function, see \cite{Kli:1991}, Proposition 2.3.3. Therefore $V^S_{K,q}$ is lower semi-continuous whenever any $u\in \L^S(\C^n)$ with $u\leq q$ on $K$ can be approximated by a continuous member of the same family. Therefore identifying lower-semicontinuous $V^S_{K,q}$ can be done by finding appropriate regularizing methods that preserve the class $\L^S(\C^n)$, see the discussion in Section 5 of \cite{MagSigSigSno:2023}.
When $S$ is a lower set, and only then, can we use the traditional regularizing method of convolution with a smoothing kernel, by \cite{MagSigSigSno:2023}, Theorem 5.8. When $S$ is assumed to contain a neighborhood of zero, the method of Ferrier approximation can be employed, see Section 3 of \cite{BayHusLevPer:2020}, so in that case $V^S_{K,q}$ is lower-semicontinuous. 

\smallskip

Bayraktar, Hussung, Levenberg and Perera 
state in Theorem 1.1 of \cite{BayHusLevPer:2020}
that for every convex set $S\subset \R^n_+$ that contains a neighborhood of $0$ and $q$ an admissible weight on a closed $E$, then $\log \Phi ^S_{E,q} =V^S_{E,q}$. Their proof is based on Proposition 4.3 in \cite{BayHusLevPer:2020} which, 
as far as we can see,  only holds if $S$ is a lower set, that is if $[0,s_1]\times \cdots \times [0,s_n]\subset S$ for every $s\in S$. A counterexample to its claim is found in \cite{MagSigSigSno:2023}, Example 7.4. 
The statement of Theorem 1.1 in \cite{BayHusLevPer:2020} is however valid
as it follows from our theorem. 

\smallskip

Previous works \cite{BayBloLev:2018, BayBloLevLu:2019, BayHusLevPer:2020} refer to this subject as \emph{Pluripotential theory and convex bodies}, as the support $S$ is throughout assumed to be a convex body, that is with non-empty interior. With the goal of relaxing the conditions on $S$ as much as possible, we study separately the case when $S$ has empty interior. 
If $S$ has empty interior but  is convex and contains zero, then $S$ lies in some subspace of dimension $\ell<n$. Therefore $S$ is the image of some linear $L$ map of some convex set $T$ in $\R^\ell$. The next theorem asserts that this linear map $L$ can be chosen so that $T$ is a convex body in $\R^\ell_+$ and how the extremal functions with respect to $S$ and $T$ are related. This result gives us a reduction  of the proof of Theorem \ref{thm:1.1} to  the case when $S$ is a convex body.

\medskip 
\begin{theorem}
\label{thm:1.2}
Let $S\subset \R^n_+$ be compact and convex with $0\in S$.
Assume that $S$ is of dimension $\ell\leq n$ and that 
$S\cap \Q^n$ is dense in $S$.  Then there exists
a linear map $L\colon \R^\ell\to \R^n$ and 
a compact convex subset $T$ of $\R^\ell_+$ with $0\in T$ 
such that $L(T)=S$, $L(e_{k})\in \Z^n$ for ${k}=1,\dots,\ell$ and 
$L^{-1}(\Z^n)=\Z^\ell$.  Let ${F}_L\colon \C^{*n} \to \C^{*\ell}$
denote the holomorphic map given by
\begin{equation}
  \label{eq:1.1}
  {F}_L(z)=(z^{L(e_1)},\dots,z^{L(e_\ell)}), \qquad z\in \C^{*n}. 
\end{equation}
Then for all $m\in \N$ there are well-defined bijective pullbacks 
\begin{equation}
  {F}_L^*\colon \mathcal{P}^T_m(\C^\ell)\to \mathcal P ^S_m(\C^n), \qquad {F}_L^*\colon \L^T(\C^\ell)\to \L^S(\C^n) \label{eq:1.2}
\end{equation}
 and for every admissible weight function $q$ on a compact subset $K$ of $\C^{n}$ and $m\in \N^*$
\begin{align}\Phi^S_{K,q,m}&={F}_L^*\big(\Phi^T_{K',q',m}\big)
  \qquad \text{on }\C^{*n}, \quad \text{and} \label{eq:1.3}\\
V^S_{K,q}&={F}_L^*\big( V^T_{K',q'}\big)  \qquad \text{on }\C^{*n},\label{eq:1.4} 
\end{align}
where $K'={F}_L(K)$ and  $q'(w)={F}_{L*}q(w)=\inf\{q(z)\,;\, z\in
{F}_L^{-1}(w)\cap K\}$. 
\end{theorem}

Magnússon, Sigurðsson and Snorrason continue the discussion in \cite{MagSigSno:2023}, where a Bernstein-Walsh-Siciak theorem on approximation by polynomials in $\mathcal{P}^S(\C^n)$ is proved.

\subsection*{Acknowledgment}  
The results of this paper are a part of a research project, 
{\it Holomorphic Approximations and Pluripotential Theory},
with  project grant 
no.~207236-051 supported by the Icelandic Research Fund.
We would like to thank the Fund for its support and the 
Mathematics Division, Science Institute, University of Iceland,
for hosting the project.   
We thank Sigurður Jens Albertsson and
Atli Fannar Franklín
for helpful discussions.  
We thank Bergur Snorrason for many helpful discussions and 
a careful reading of the manuscript.   We thank the referree for asking the correct questions, leading to improvements of the paper.

\section{Convex sets $S$ of lower affine dimension}
\label{sec:02}

If $S\subset \R^n$ is
convex with $0\in S$ and has empty interior then the smallest subspace $W$ of $\R^n$ that
contains $S$ is a proper subspace of $\R^n$.  For any linear map
$L\colon \R^\ell\to \R^n$ whose image is $W$ the set $T=L^{-1}(S)$ is 
convex with non-empty interior.   Before we make an explicit
choice of the linear map $L$, we show that in general $L$ induces  a
map between the polynomial rings ${\mathcal P}^T(\C^\ell)$
and ${\mathcal P}^S(\C^n)$ and the Lelong classes
$\L^T(\C^\ell)$ and $\L^S(\C^n)$.

\smallskip
Let $S$ be a compact convex subset of $\R^n_+$ with $0\in S$.
Let $L\colon \R^\ell\to \R^n$ be a linear map and assume that
$S=L(T)$, where $T$ is a compact convex subset of $\R^\ell_+$ with 
$0\in T$.  The relation between $\varphi_S$ and $\varphi_T$ is
given by 
\begin{equation}
  \label{eq:2.1}
  \varphi_S(\xi)=\sup_{t\in T}\scalar{L(t)}\xi
=\sup_{t\in T}\scalar t{L^*(\xi)}=\varphi_T(L^*(\xi)),
\end{equation}
where $L^*\colon \R^n\to \R^\ell$ denotes the adjoint  of $L$.
If $\eta_{k}=L(e_{k})\in \Z^n$ for ${k}=1,\dots,\ell$ then we have 
a well defined rational  map
\begin{equation}
  \label{eq:2.2}
{F}_L\colon \C^{*n}\to \C^{*\ell}, \qquad {F}_L(z)=(z^{\eta_1},\dots,z^{\eta_\ell}).
\end{equation}
Observe that for every $t\in \R^\ell$ and $z\in \C^{*n}$ we have
\begin{equation}
  \label{eq:2.3}
  \scalar t{L^*(\Log z)}
=\sum_{j=1}^n\sum_{{k}=1}^\ell t_{k} \eta_{{k},j}\log|z_j|
=\scalar t{\Log\, {F}_L(z)},
\end{equation}
so by  (\ref{eq:2.1}) we have 
\begin{equation}
  \label{eq:2.4}
  H_S(z)=\varphi_T(\Log({F}_L(z)))
=H_T({F}_L(z)), \qquad z\in \C^{*n}.
\end{equation}
By continuty this gives  a relation between $H_S$ and $H_T$,
\begin{equation}
  \label{eq:2.5}
  H_S={F}_L^*H_T
\end{equation}
and a pullback
\begin{equation}
  \label{eq:2.6}
    {F}_L^*\colon \L^T(\C^\ell)\to \L^S(\C^n).
\end{equation}
If $p\in {\mathcal P}^T_m(\C^\ell)$ with
$p(w)=\sum_{\beta\in mT}b_\beta w^\beta$,  then
\begin{equation}
  \label{eq:2.8}
  {F}_L^*p(z)=\sum_{\beta\in mT}b_\beta \big({F}_L(z)\big)^\beta
=\sum_{\beta\in mT}b_\beta\big(z^{\beta_1L(e_1)}\cdots 
z^{\beta_\ell L(e_\ell)}\big)
=\sum_{\beta\in mT}b_\beta z^{L(\beta)}.
\end{equation}
Since $L(mT)=mS$ this equation shows that 
we have a pullback
\begin{equation}
  \label{eq:2.9}
 {F}_L^*\colon {\mathcal P}^T_m(\C^\ell)  \to  {\mathcal P}^S_m(\C^n) 
\end{equation}
which is surjective if $L(mT\cap \Z^\ell)=mS\cap \Z^n$.

\medskip
We need an appropriate weight on ${F}_L(K)\subset \C^\ell$. 
We have natural definitions of push-forwards $F_*$ of the classes
$\USC(K)$ and $\LSC(K)$ of upper and lower semi-continuous functions
by a continuous map $F\colon K\to \C^\ell$
such that $F_*\colon \USC(K)\to \USC(F(K))$ and $F_*\colon \LSC(K)\to
\LSC(F(K))$, but their definitions are different. The function
$w\mapsto \sup\{u(z)\,;\, z\in F^{-1}(w)\}$ 
is in $\USC(F(K))$ for every $u\in \USC(K)$
and $w\mapsto \inf\{u(z)\,;\, z\in F^{-1}(w)\}$ 
is in $\LSC(F(K))$ for every
$u\in \LSC(K)$.  In this context it is natural to use the second
definition of $F_*$:

\begin{proposition}
  \label{prop:2.1}
Let $q\colon K\to \R\cup\{+\infty\}$ be an admissible weight
on a compact subset $K$ of $\C^n$ and  $F\colon U\to \C^\ell$ be an
open holomorphic map defined on some neighborhood of $K$.  
Then  the push-forward $F_*q\colon F(K)\to 
\R\cup\{+\infty\}$ of $q$ by $F$, given as 
\begin{equation*}
  F_*q(w)=\inf\{q(z) \,;\, z\in F^{-1}(w)\cap K\}, \qquad
  w\in F(K),
\end{equation*}
is an admissible weight on $F(K)$.
\end{proposition}

\begin{proof}  Take $w\in F(K)$ and  $w_j\to w$ such that
$\lim_{j\to \infty} F_*q(w_j)=\varliminf_{\zeta\to w}F_*q(\zeta)$
and $\varepsilon_j\searrow 0$.  Then there exists
$z_j\in F^{-1}(w_j)\cap K$ such that $q(z_j)\leq F_*q(w_j)+\varepsilon_j$.  
Let $(z_{j_k})$ be a subsequence of $(z_j)$ converging to $z\in K$.  
Since $q$ is lower semi-continuous we have  
\begin{align*}
F_*q(w)&\leq q(z) \leq \varliminf_{k\to \infty} q(z_{j_k})
\leq \lim_{k\to \infty} F_*q(w_{j_k}) +\varepsilon_{j_k}=\varliminf_{\zeta\to w}F_*q(\zeta).
\end{align*}
Hence $F_*q$ is lower semi-continuous on $F(K)$.
Let $A=\{z\in K\,;\, q(z)<+\infty\}$. 
Since $A$ is non-pluripolar and $F$ is open, it follows that
$F(A)  \subseteq \{w\in F(K) \,;\,
F_*q(w)<+\infty\}$ are non-pluripolar. Hence $F_*q$ is admissible. 
\end{proof}

\bigskip

Let $K\subseteq \C^n$ be compact and $q$ an admissible weight on $K$.
For ease of notation, write $K'={F}_L(K)$ and $q'={F}_{L*}q$. We have ${F}_L^*q'
\leq q$, so if $v\in \PSH(\C^\ell)$ and
$v\leq q' 
$ on $K'$, 
then $u={F}_L^*v$ satisfies $u\leq q$ on $K$. Hence
\begin{equation}
  \label{2.7}
{F}_L^*\big( V^T_{K',q'}\big) 
\leq V^S_{K,q}  
\end{equation}
and equality holds if (\ref{eq:2.6})
is surjective. If $p \in\mathcal P^S_m (\C^n)$ and $\log|p|^{1/m}\leq q'$  
then  $\log|{F}_L^*p|^{1/m}\leq q$ on $K$.  Hence
\begin{equation}
  \label{eq:2.10}
  {F}_L^*\big(\Phi^T_{K',q',m}\big) \leq \Phi^S_{K,q,m},
  \qquad m\in \N,
\end{equation}
and equality holds if the pullback (\ref{eq:2.9}) is surjective.

\section{Preliminary algebra}
\label{sec:03}

In order to prove Theorem \ref{thm:1.2}
we show that there exists
a linear map $L\colon \R^\ell\to \R^n$ which maps 
a compact convex subset $T$ of $\R^\ell_+$ with $0\in T$ onto $S$
such that the two pullbacks (\ref{eq:2.6}) and (\ref{eq:2.9})
are surjective.  For that purpose we need some preliminary results.

\begin{lemma}
  \label{lem:3.1}
Let  $W$ be a subspace of $\R^n$ of dimension $\ell$ spanned by
vectors in $\Q^n$. Then there exists a linear map $L\colon \R^\ell\to
\R^n$  such that  $L(\R^\ell)=W$,  $L(\Z^\ell)\subseteq \Z^n$, $L^{-1}(\Z^n)=\Z^\ell$,
and $L^{-1}(\R^n_+)\subseteq \R^\ell_+$.
\end{lemma}

\begin{proof}  The group $W\cap \Z^n$ is a subgroup of $\Z^n$ and is
therefore isomorphic to $\Z^\nu$ for some $\nu\leq n$. It is
generated by $\nu$ elements $g_1,\dots,g_\nu$ which
are independent, so $\nu\leq \ell$.  In order to prove that 
$\nu=\ell$ we observe first that since 
$W$  is spanned by vectors in $\Q^n$ it is sufficient to 
show that  that $g_1,\dots,g_{\nu}$ span $W\cap \Q^n$ over $\Q$.
Take $x\in W\cap \Q^n$ and $\lambda\in \N^*$
such that  $\lambda x\in \Z^n$.  Then $\lambda x\in W\cap \Z^n$ which
gives  $\lambda x=x_1g_1+\cdots+x_{\nu} g_{\nu}$,  with $x_{k}\in \Z$.
Hence  $x=(x_1/\lambda)g_1+\cdots+(x_{\nu}/\lambda)g_{\nu}$ and $\nu=\ell$.  

\smallskip

Let $M\colon \R^\ell\to \R^n$ be defined by $M(e_k)=g_k$, $k=1,\dots,\ell$
and let $\Gamma=M^{-1}(\R^n_+)$.  
Since the generators are linearly independent $M$ is injective 
and consequently the adjoint $M^*$ is surjective.   
Hence we can find indices  $j_1,\dots,j_\ell$ such that
$a_{k}=M^*(e_{j_{k}})$, ${k}=1,\dots,\ell$ span $\R^\ell$.
The formula for $M^*$ is $M^*(y)=(\scalar{g_1}y,\dots,\scalar{g_\ell}y)$
so $a_{k}\in \Z^\ell$ for ${k}=1,\dots,\ell$. 
The dual cone $\Gamma^\circ=\{y\in \R^\ell\,;\, \langle y,x\rangle\geq 0\text{ for all }x\in \Gamma\}$ contains $a_k$ because if $x\in \Gamma$ then
$Mx\in \R^n_+$ and 
$\scalar x{a_{k}}=\scalar{Mx}{e_{j_{k}}}\geq 0$. Since the vectors
$a_{k}$ are linearly independent, $\Gamma^\circ$ is the closed
cone spanned by $a_1,\dots,a_\ell$. 

\smallskip

Define
$P(\xi_1,\dots,\xi_\ell)=\big\{\sum_{{k}=1}^\ell \lambda_{k}\xi_{k} 
\,;\, \lambda_{k}\in [0,1[, {k}=1,\dots,\ell\big\}$
for $\xi_1,\dots,\xi_\ell\in \R^\ell$.  
We will now show that there exist linearly independent
$\xi_1,\dots,\xi_\ell\in \Z^\ell\cap \Gamma^\circ$ such that  
$P(\xi_1,\dots,\xi_\ell)\cap \Z^\ell=\{0\}$. If $P(a_1,\dots,a_\ell)\cap \Z^\ell=\{0\}$ we are done.
Assume that  $0\neq x\in P(a_1,\dots,a_\ell)\cap \Z^\ell$,
$x=\lambda_1a_1+\cdots+\lambda_{\ell} a_{\ell}$ with 
$\lambda_{k}\in [0,1[$ and $\lambda_m\neq 0$.  
We form a new sequence $a_1^{(1)},\dots,a_\ell^{(1)}$ in $\Z^\ell$ 
by setting  $a_{k}^{(1)}=a_{k}$ for ${k}\neq m$ and
$a_m^{(1)}=x$.  After renumbering $a_1,\dots, a_\ell$  we may
assume that $m=1$.

The vectors $a_1^{(1)},\dots,a_\ell^{(1)}$ are linearly independent.
In fact, if $\mu_1a^{(1)}_1+\cdots+\mu_\ell a^{(1)}_\ell=0$ for $\mu_{k} \in
\R$, then  
$
\mu_1\lambda_1a_1+(\mu_1\lambda_2+\mu_2)a_2+\cdots+
(\mu_1\lambda_\ell+\mu_\ell)a_\ell=0$
holds, and since $a_1,\dots,a_\ell$ are linearly independent and
$\lambda_1\neq 0$ we have $\mu_1=\cdots=\mu_\ell=0$.

\smallskip

Next we show that $\#\big( P(a_1^{(1)},\dots,a_\ell^{(1)})
\cap \Z^\ell\big) <\# \big(P(a_1,\dots,a_\ell)\cap \Z^\ell\big)$ by
constructing an injective map 
$G\colon  P(a_1^{(1)},\dots,a_\ell^{(1)})\cap \Z^\ell \to 
P(a_1,\dots,a_\ell)\cap \Z^\ell$  which does not take the value $x$.
If $y\in P(a_1^{(1)},\dots,a_\ell^{(1)})\cap \Z^\ell$, then $y=\sum_{{k}=1}^\ell y_\ell a^{(1)}_{k}=y_1\lambda_1a_1+
\sum_{{k}=2}^\ell (y_1\lambda_{k} +y_\ell)a_{k}$ with $y_{k}\in [0,1[$.  
Since $a_{k}\in \Z^\ell$ the map $G$ is well defined by
$$
G(y)=y_1\lambda_1a_1+\sum_{{k}=2}^\ell 
\big(y_1\lambda_{k} +y_{k}-\lfloor y_1\lambda_{k} +y_{k}\rfloor\big)a_{k}.
$$
If $G(y)=G(z)$, then
$y_1\lambda_1=z_1\lambda_1$ and 
$y_1\lambda_{k} +y_{k}-\lfloor y_1\lambda_{k} +y_{k}\rfloor=
z_1\lambda_{k} +z_{k}-\lfloor z_1\lambda_{k} +z_{k}\rfloor$
for ${k}=2,\dots,\ell$.
Since $\lambda_1\neq 0$ we have $y_1=z_1$ and consequently
$$
y_{k}-\lfloor y_1\lambda_{k} +y_{k}\rfloor=
z_{k}-\lfloor z_1\lambda_{k} +z_{k}\rfloor, 
\quad {k}=2,\dots,\ell.
$$
The two integer parts in this equation are $0$ or $1$. They
must be the same number,  for otherwise we would have a strictly
negative value on one side and a positive value on the other side.
Hence $y_{k}=z_{k}$ for ${k}=2,\dots,\ell$ and  $G$ is injective.  
The point $x=a^{(1)}_1\in
P(a_1,\dots,a_\ell)\cap \Z^\ell$ is not in the image of $G$, 
for $x=G(y)$ would imply that $y_1\lambda_1=\lambda_1$ for
some $y_1\in [0,1[$.
Hence  
$$\#\big(P(a_1^{(1)},\dots,a_\ell^{(1)})
\cap \Z^\ell\big) <\#\big(P(a_1,\dots,a_\ell)\cap \Z^\ell\big).$$
Observe that $a^{(1)}_1,\dots,a^{(1)}_\ell$ are in 
$\Gamma^\circ$ which is the closed convex cone spanned by 
$a_1,\dots,a_\ell$. By repeating this argument with $a^{(1)}_1,\dots,a^{(1)}_\ell$
in the role of $a_1,\dots,a_\ell$ finitely many times we conclude
that there exist linearly independent vectors $\xi_1,\dots,\xi_\ell$ in
$\Z^\ell\cap \Gamma^\circ$ such that  $P(\xi_1,\dots,\xi_\ell)\cap \Z^\ell=\{0\}$. 

\smallskip

We define the linear map $B\colon \R^\ell\to \R^\ell$ by
$B(x)=(\scalar{\xi_1}x,\dots,\scalar{\xi_\ell}x)$ and observe
that $B^*(x)=x_1\xi_1+\cdots+x_\ell\xi_\ell$.  This gives that
$B^*(\Z^\ell)\subseteq \Z^\ell$ and since $P(\xi_1,\dots,\xi_\ell)\cap
\Z^\ell=\{0\}$ it follows that $B^*(\Z^\ell)=\Z^\ell$.  We have
$$
Bx \in \Z^\ell \ \Leftrightarrow \  
\scalar {Bx}y\in \Z^\ell \ \forall  y\in \Z^\ell
\ \Leftrightarrow \ 
\scalar {x}{B^*y}\in \Z^\ell \ \forall  y\in \Z^\ell
\  \Leftrightarrow \  
x\in \Z^\ell.
$$
Finally, we define the linear map $L=M\circ B^{-1}$.
Since $M$ and $B$ are injective and $W$ is the image of $M$ 
we have  $L(\R^\ell)=W$, 
$$L(\Z^\ell)=M(B^{-1}(\Z^\ell))=M(\Z^\ell)=W\cap \Z^n.
$$
Any $y\in W\cap \R^n_+$ has $x=M^{-1}(y)\in \Gamma$ and since
$\xi_{k}\in \Gamma^\circ$ we have $\scalar x{\xi_{k}}\geq 0$
for ${k}=1,\dots,\ell$, which shows that  $B(x)\in \R^\ell_+$.
Thus $L^{-1}(\R^n_+)=B(M^{-1}(\R^n_+))\subseteq \R^\ell_+$. 
\end{proof}

\begin{lemma}
  \label{lem:3.2} If $L\colon \R^\ell\to \R^n$ is an injective linear
  map such that 
 $L(\Z^\ell)\subseteq \Z^n$ and $L^{-1}(\Z^n)=\Z^\ell$, then 
$L^*(\Z^n)=\Z^\ell$, where $L^*$ denotes the adjoint of $L$.  
\end{lemma}

\begin{proof}  Let $A$ be the matrix for $L$ with respect to the
  standard bases on $\R^\ell$ and $\R^n$ and 
let  $D=SAT$ be the Smith normal form
of $A$, see Hungerford \cite{Hun:1974}, Propostion 2.11, p.~339.   
This means that we factor $A$ into 
$A=S^{-1}DT^{-1}$,  with $S\in \Z^{n\times n}$ 
such that  $x\mapsto Sx$ is  invertible on $\Z^n$,  
$T\in \Z^{\ell\times \ell}$ such that $x\mapsto Tx$
is invertible on $\Z^\ell$ and $D=SAT$ is an $n\times \ell$
matrix with all entries zero exccept 
the diagonal entries $a_{j,j}$ which are  of the
form $a_{j,j}=\pm d_j/d_{j-1}\neq 0$ where 
$d_j$ is the greatest common divisor of all determinants 
of $j\times j$  minors of $L$.

Since $D$ only maps integer vectors into $\Z^n$ and
$D(e_j/a_{j,j})=e_j\in \Z^n$ it follows that $a_{j,j}=\pm 1$. Hence
$D(\Z^\ell)\subseteq \Z^n$ and $D^{-1}(\Z^ n)=\Z^\ell$. The matrix $D^*=T^*A^*S^*$ is the normal form of $A^*$ and since the
first $k$ columns of $D^*$ are $\pm e_j$ we 
see that $L^*(\Z^n)=\Z^\ell$.

\end{proof}

\section{Surjectivity of the pullback ${F}_L^*$}
\label{sec:04}

\noindent 
Let us continue the study of the map ${F}_L$ now assuming
that $L\colon \R^\ell \to \R^n$ satisfies the conditions in Lemma \ref{lem:3.1}
with $W=L(\R^\ell)=\Span_\R(S)$ and $T=L^{-1}(S)$.  Since $L^{-1}(\R^n_+)\subseteq \R^\ell_+$ and $L(\Z^\ell)\subseteq \Z^n$ 
we have $L(mT\cap \Z^\ell)=mS\cap \Z^n$ and it follows that
the pullback (\ref{eq:2.9}) is surjective. Recall that $\eta_k=L(e_k)$
for $k=1,\dots,\ell$. Our goal is then to
prove that the pullback in 
(\ref{eq:2.6}) is  surjective.

\smallskip
We observe that 
if $A$ is an $m\times n$ matrix with entries in $\Z$, $a_1,\dots,a_m$
are its line vectors, $A_1,\dots,A_n$ are its column vectors and 
$b\in \Z^n$ is a column vector, then for every 
$t\in \C^{*m}$ we have 
\begin{equation}
\label{eq:4.1}
(t^{A_1},\dots,t^{A_n})^b=t^{Ab}.  
\end{equation}
In order to show that  ${F}_L\colon \C^{*n} \to \C^{*\ell}$ is
surjective we take $w\in \C^{*\ell}$, write $w_j=e^{b_j}$ for some 
$b_j\in \C$,  and let $A$ be the $\ell\times n$ matrix with line vectors
$\eta_1,\dots,\eta_\ell$, i.e.~$A$ is the matrix for $L^*$ with respect
to the standard bases on $\R^n$ and $\R^\ell$. 
By Lemma \ref{lem:3.2}  $L^*$ is surjective,   
so there is a solution $c\in \C^n$ of the linear system $Ac=b$.  
If we set $z_j=e^{c_j}$, then
$$
{F}_L(z)=(z^{\eta_1},\dots,z^{\eta_\ell})
=(e^{\scalar{\eta_1}c},\dots,e^{\scalar{\eta_\ell}c}) 
=(e^{b_1},\dots,e^{b_\ell})=w.
$$
which shows that $F_L$ is surjective.

\medskip

For every $z\in \C^{*n}$ we have
\begin{equation}
  \label{eq:4.5}
  \dfrac{\partial {F}_{L,{k}}}{\partial z_j}(z)
=\dfrac{\partial}{\partial z_j}
\big(z_1^{\eta_{{k},1}}\cdots z_n^{\eta_{{k},n}}\big)
= \frac{\eta_{{k},j} {F}_{L,{k}}(z)}{z_j} {}_.
\end{equation}
Since $L$ is injective, the vectors
$\eta_1,\dots,\eta_\ell$ are linearly independent, so 
the matrix 
\begin{equation*}
  \left[
    \begin{matrix}
      \eta_{1,1} &\dots&\eta_{1,n}\\ 
\vdots & \ddots & \vdots\\
      \eta_{\ell,1} &\dots&\eta_{\ell,n}
    \end{matrix} \right]
\end{equation*}
has rank $\ell$.  We choose $1\leq j_1<\cdots<j_\ell\leq n$ such that
the columns number $j_1,\dots,j_\ell$ are linearly independent.  Then 
the determinant of the corresponding $\ell\times \ell$ 
submatrix of the Jacobi matrix of ${F}_L$ is
$$
  \left|    \begin{matrix}
      \eta_{1,j_1} {F}_{L,1}(z)/z_{j_1}&\dots&\eta_{1,j_\ell}{F}_{L,1}(z)/z_{j_\ell}\\ 
\vdots & \ddots & \vdots\\
      \eta_{\ell,j_1}{F}_{L,\ell}(z)/z_{j_1} &\dots&\eta_{\ell,j_\ell}{F}_{L,\ell}(z)/z_{j_\ell} 
    \end{matrix} \right|
=\dfrac{{F}_{L,1}(z)\cdots {F}_{L,\ell}(z)}
{z_{j_1}\cdots z_{j_\ell}}
  \left|
    \begin{matrix}
      \eta_{1,j_1} &\dots&\eta_{1,j_\ell}\\ 
\vdots & \ddots & \vdots\\
      \eta_{\ell,j_1} &\dots&\eta_{\ell,j_\ell}
    \end{matrix} \right| \neq 0.
$$

\medskip

Let $w_0=F_L(z_0)$.
By the implicit function theorem the fiber 
${F}_L^{-1}(w_0)$ through  $z_0\in \C^{*n}$
is a complex manifold of dimension $n-\ell$. We can parametrize the fibers of ${F}_L$.  In order to do
so we extend $\eta_1,\dots,\eta_\ell$ to a basis in $\R^n$ by choosing 
$\eta_{\ell+1},\dots,\eta_n$ as a  basis for
$W^\perp=\Ker\, L^*$.  Since $W$ is spanned by vectors in $\Q^n$ the same holds for
$W^\perp$, so 
it is always possible to choose $\eta_j$ from $\Z^ n$.
We let  $M\colon \R^{n-\ell}\to \R^n$ be the linear map 
with $M(e_{j})=\eta_j$, for $j=\ell+1,\dots,n$.
Here we identify $\R^{n-\ell}$ with the subspace $\{0\}\times \R^{n-\ell}$
in $\R^n$ and view the vectors $e_{\ell+1},\dots,e_n$ in the standard
basis in $\R^n$ as the standard basis for $\R^{n-\ell}$.
Furthermore, we write $t=(t',t'')\in \C^n$ with $t'\in \C^\ell$ and
$t''\in \C^{n-\ell}$.  We let  $N\colon \R^n \to \R^n$ be the linear
map defined by $N(t)=L(t')+M(t'')$. Then 
\begin{equation}
  \label{eq:4.2}
 N^*(x)=(\scalar{\eta_1}x,\cdots,\scalar{\eta_n}x) = (L^*(x),M^*(x)). 
\end{equation}
If $B$ is the matrix for $N^*$ with respect to the standard basis,
then $\eta_1,\dots,\eta_n$ are the line vectors of $B$, 
so if we let $B_1,\dots,B_n$ be the column vectors in $B$,
then for every $b\in \Z^n$ we have
$$
(z_1t^{B_1},\dots,z_nt^{B_n})^b=z^bt^{Bb}.
$$
If we apply this formula with $b=\eta_j^*$ for $j=1,\dots,\ell$, then
we see that 
$$
B\eta_j^*=(\scalar{\eta_1}{\eta_j},\dots,\scalar{\eta_\ell}{\eta_j},
0,\dots,0)
=(L^*(e_j),0)=(A_j^*,0)
$$
and consequently 
$$
{F}_L(z_1t^{B_1},\dots,z_nt^{B_n})
=(z^{\eta_1}t^{B\eta_1^*},\dots,
z^{\eta_\ell}t^{B\eta_\ell^*})
=(z^{\eta_1}t'^{A_1},\dots,z^{\eta_\ell}t'^{A_\ell}).
$$
This formula tells us that every point
of the form $(z_1t^{B_1},\dots,z_nt^{B_n})$
with 
$$
(t'^{A_1},\dots,t'^{A_\ell})=(1,\dots,1)
$$ 
is in the fiber of ${F}_L$ through  $z$. Our observations lead to the
following lemma.

\begin{lemma}
  \label{lem:4.1} 
Let $B\in \Z^{n\times n}$ be the matrix with line vectors  
$\eta_1,\dots,\eta_n$ and let
$B_1,\dots,B_n$ denote its column vectors.  Let $z\in \C^{*n}$ 
and define 
\begin{equation}
  \label{eq:4.3}
  \Xi_z\colon \C^{*n}\to \C^n, \qquad
\Xi_z(t)=(z_1t^{B_1},\dots,z_nt^{B_n}), \quad t\in \C^{*n}.
\end{equation}
Then 
\begin{equation}
  \label{eq:4.4}
\Upsilon_z\colon \C^{*(n-\ell)}\to \C^n, \qquad
\Upsilon_z(t) =\Xi_z(1,\dots,1,t''),
\quad t''\in \C^{*(n-\ell)},  
\end{equation}
is a parametrization of the fiber of ${F}_L$ through $z$.
\end{lemma}  

\begin{proof} 
From our calculations before the statement of the 
lemma it follows that we only need to prove that every $\zeta\in \C^{*n}$ 
which satisfies ${F}_L(\zeta)={F}_L(z)$ is of the
form $\zeta=(z_1t^{B_1},\dots,z_nt^{B_n})$ for some
$t=(1,\dots,1,t'')$ with $t''\in \C^{*(n-\ell)}$.

\smallskip

We choose $s=(s_1,\dots,s_n)\in \C^n$ such that
$\zeta_j=z_je^{s_j}$ for $j=1,\dots,n$ and observe that
$\zeta^{\eta_{k}}=z^{\eta_{k}}$ implies that
$e^{\scalar{\eta_{k}}s}=1$ for ${k}=1,\dots,\ell$.
Hence there exists $\mu=(\mu_1,\dots,\mu_\ell)\in \Z^\ell$ 
such that $\scalar{\eta_{k}}s=2\pi i \mu_{k}$ for ${k}=1,\dots,\ell$. 
By Lemma \ref{lem:3.2} there exists $\nu\in \Z^n$ such that
$L^*(\nu)=\mu$.  We have
$L^*(s)=(\scalar{\eta_1}s,\dots,\scalar{\eta_\ell}s)$, so
$$
s-2\pi i \nu\in \Ker_\C \, L^*=\Span_\C\, \{\eta_{\ell+1},\dots,\eta_n\},
$$
where $\Ker_\C\, L^*$ is the kernel of the natural extension of $L^*$ to a
$\C$-linear map $\C^n\to \C^\ell$.   We take $u''=(u_{\ell+1},\dots,u_n)\in
\C^{n-\ell}$ such that $s-2\pi i\nu=u_{\ell+1}\eta_{\ell+1}+\cdots+u_n\eta_n$,
set $u=(0,\dots,0,u'')\in \C^n$ and $t=(t_1,\dots,t_n)\in \C^n$ 
with $t_j=e^{u_j}$ for $j=1,\dots,n$.   Then
$$
z_jt^{B_j}=z_je^{u_{\ell+1}\eta_{\ell+1,j}+\cdots+u_n \eta_{n,j}}
=z_je^{s_j-2\pi i\nu_j}=\zeta_j.
$$ 
Hence  $\zeta=\Upsilon_z(t'')$ and $\Upsilon_z$ is
a parametrization of the fiber of ${F}_L$ through $z$.  
\end{proof}

\newpage

\noindent\textbf{End of proof of Theorem \ref{thm:1.2}:} Take $u\in \L^S(\C^n)$, fix $z\in \C^n$
and define $u_z\in \PSH(\C^{*(n-\ell)})$ by 
$$
u_z(t)=u(\Upsilon_z(t)), \qquad t\in \C^{*(n-\ell)}.
$$
Then we have the estimate 
$$
u_z(t) \leq c_u+H_S(\Upsilon_z(t))
=c_u+H_T({F}_L(\Upsilon_z(t)))=c_u+H_T({F}_L(z))=c_u+H_S(z).
$$
This gives us extension of $u_z$ to a plurisubharmonic function
on $\C^{n-\ell}$ by the formula 
$$
u_z(t)=\varlimsup_{\C^{*(n-\ell)}\ni \tau \to t} u_z(\tau), 
\qquad t\in \C^{n-\ell}.
$$
We have $u_z(t)\leq c_u+H_S(z)$ for every $t$, so by the 
Liouville theorem for plurisubharmonic functions  
$u_z$ is the constant function taking the value $u(z)$.

\medskip

Define $v\colon \C^{*\ell}\to
\R\cup \{-\infty\}$ by $v(w)=u(z)$ for any $z\in F_L^{-1}(w)\subset \C^{*n}$. This is well-defined since $F_L$ is surjective and $u$ is constant on $F_L^{-1}(w)$. Since $F_L$ is a surjective submersion of constant rank there exists for every $w_0\in \C^{*\ell}$ an open neighborhood $U_0$ of $w_0$ 
and a holomorphic map $s\colon U_0\to \C^n$ such that 
$s(w_0)=z_0$ and ${F}_L(s(w))=w$ for every $w\in U_0$.  Then $v(w)= u(s(w))$ on $U_0$, which shows that $v\in \mathcal{PSH}(\C^{*\ell})$. 

\medskip

We have the estimate 
$v(w)=u(z)\leq c_u+H_S(z)=c_u+H_T(w)$ for every $w\in \C^{*\ell}$
which implies that $v$ extends to a function in $\L^T(\C^\ell)$
satisfying $u={F}_L^*v$.    Hence the pullback 
(\ref{eq:2.6}) is surjective. \hfill $\square$

\bigskip

\begin{corollary}\label{cor:4.2}
  Let $S\subset \R^n_+$ be compact and convex with $0\in S$ and assume
  that $S$ has empty interior and that  $S\cap \Q^n$ is dense in
  $S$. Then $V^{S*}_{K,q}$ is maximal.
\end{corollary}

\begin{proof}
By Theorem \ref{thm:1.2} we have $V^{S*}_{K,q}=V^{T*}_{K',q'}\circ
{F}_L$ on $\C^{*n}$. From Lemma \ref{lem:4.1} follows that for every $z\in
\C^{*n}$ we have $V^{S*}_{K,q}\circ
\Upsilon_z=V^{T*}_{K',q'}\circ{F}_L\circ \Upsilon_z$. Since
${F}_L\circ \Upsilon_z$ is constant it follows that $V^{S*}_{K,q}$ is maximal on
$\C^{*n}$. Since $V^{S*}_{K,q}$ is locally bounded and
$\C^n\setminus\C^{*n}$ is pluripolar, $V^{S*}_{K,q}$ is maximal on
$\C^n$. 
\end{proof}

\section{Solutions of Cauchy-Riemann equations}  
\label{sec:05}

\noindent
To construct functions in ${\mathcal P}^S_m(\C^n)$ 
we are going to apply a special case  of Hör\-mander's theorem 
on the existence of solutions of the Cauchy-Riemann equation
with $L^2$ estimates.  Before we can state it  we need to introduce
some notation.  Let $\varphi\colon X\to \overline \R$ be a measurable
function on an open subset $X$ 
of $\C^n$  and let $L^2(X,\varphi)$ denote the set of all 
functions $u\in \Ltwoloc(X)$  such that
\begin{equation}
  \label{eq:5.1}
    \int_X |u|^2e^{-\varphi}\, d\lambda <+\infty,
\end{equation}
where $\lambda$ denotes the Lebesgue measure.
Then $L^2(X,\varphi)$  is a Hilbert space with inner product
\begin{equation}
  \label{eq:5.2}
  \scalar uv_\varphi =\int_X u\bar v\, e^{-\varphi} d\lambda, 
\qquad u,v\in L^2(X,\varphi),
\end{equation}
and corresponding norm $\|u\|_\varphi =\scalar uu_\varphi^{1/2}$ for $u\in L^2(X,\varphi)$.

\medskip

Similarly,  we let  $L^2_{(0,1)}(X,\varphi)$ denote the set of all
$(0,1)$-forms $f=f_1\, d\bar z_1+\cdots+f_n\, d\bar z_n$,
with coefficients  $f_j\in L^2(X,\varphi)$.
Then $L^2_{(0,1)}(X,\varphi)$  is a Hilbert space with inner product
\begin{equation}
  \label{eq:5.4}
  \scalar fg_\varphi =\int_X 
\scalar fg  e^{-\varphi} d\lambda, \quad  
\scalar fg = f_1\bar g_1+\cdots+\,f_n\bar g_n,
\quad f,g\in L^2_{(0,1)}(X,\varphi),
\end{equation}
and corresponding norm $\|f\|_\varphi =\scalar ff_\varphi^{1/2}$ for $f\in L^2_{(0,1)}(X,\varphi)$.

\medskip

The function $u\in \Ltwoloc(X)$ is 
a solution of the Cauchy-Riemann equation  
$\bar\partial u=f$, i.e.~$\bar{\partial}_ju=f_j$ for $j=1,\dots,n$,  
in the sense of distributions if
\begin{equation}
  \label{eq:5.6}
-\int_X u \bar \partial_j v \, d\lambda
=\int_X f_jv\, d\lambda, \qquad v\in C_0^\infty(X).
\end{equation}
A necessary condition for existence of solution is that
$$
\bar \partial f=\sum_{k=1}^n \bar \partial f_k\wedge d\bar z_k
=\sum_{j<k}\big(\bar\partial_k f_j
-\bar \partial_j f_k\big) d\bar z_j\wedge d\bar z_k=0.  
$$

\medskip
\begin{theorem} {\bf(H\"ormander)} \label{thm:5.1} 
  Let $X$ be a pseudoconvex domain  of  
$\C^n$, $\varphi\in \PSH(X)$ and define for $a\in \R$,  
  \begin{equation}
    \label{eq:5.7}
    \varphi_a(z)=\varphi(z)+a\log(1+|z|^2), \qquad z\in X.
  \end{equation}
Then for every $a>0$ and $f\in L^2_{(0,1)}(X,\varphi_{a-2})$ 
satisfying $\bar\partial f=0$ 
there exists a solution 
$u\in L^2(X,\varphi_a)$ of 
$\bar\partial u=f$ satisfying the estimate
\begin{align}  
\label{eq:5.8}
\|u\|_{\varphi_a}^2&=\int_X |u|^2(1+|z|^2)^{-a} e^{-\varphi}\,
                     d\lambda \\
&\leq 
\dfrac 1a\int_X |f|^2(1+|z|^2)^{-a+2} e^{-\varphi}\, d\lambda
=\dfrac 1a \|f\|_{\varphi_{a-2}}^2.
\nonumber
\end{align}
If $f_j\in \mathcal{C}^\infty(X)$ for $j=1,\dots,n$, then $u\in \mathcal{C}^\infty(X)$. 
\end{theorem}

For a proof see Hörmander \cite{Hormander:convexity}, Theorem 4.2.6.
The smoothness statement follows from the fact that the Laplace
operator  $\Delta=4\sum_{j=1}^n \partial_j\bar \partial_j$
is is hypoelliptic, i.e.~every distribution 
solution $v\in {\mathcal D}'(X)$ of $\Delta v=g$ with
$g\in \mathcal{C}^\infty(X)$ is  a $\mathcal{C}^\infty$ function on $X$.
If $u$ satisfies  $\bar \partial u=f$
in the sense of distributions, then
\begin{equation}
  \label{eq:5.9}
  \Delta u
=4\sum_{j=1}^n\partial_j\bar\partial_j  u
=4\sum_{j=1}^n\partial_j f_j\in \mathcal{C}^\infty(X),
\end{equation}
so $u\in \mathcal{C}^\infty(X)$. 
Observe  that if $e^{-\varphi}$ is not integrable in any
neighborhood of $z\in X$, then  a solution $u\in L^2(X,\varphi_a)$
satisfying (\ref{eq:5.8}) vanishes at $z$.  

\smallskip
Let us now review a simple method 
for estimating  values $u(z)$ of a solution $u$ of $\bar \partial u=f$ 
satisfying (\ref{eq:5.8}) for some measurable $\varphi\colon X\to
\overline \R$.   
If   $\bar B(z,\delta) \subseteq X\setminus \supp f$,
then $u$ is holomorphic in a neighborhood of $\bar B(z,\delta)$ 
and the mean value theorem gives  
\begin{equation}
  \label{eq:5.10}
u(z)={\mathcal M}_\delta u(z)=
\dfrac 1{\Omega_{2n}\delta^{2n}}
\int_{B(z,\delta)} u\, d\lambda. 
\end{equation}
By the Cauchy-Schwarz inequality where $\Omega_{2n}$ denotes the volume of the unit ball
\begin{align}
  \label{eq:5.11}
  |u(z)|&\leq \dfrac 1{\Omega_{2n}\delta^{2n}} 
\int_{B(z,\delta)}|u|e^{-\varphi_a/2}\cdot e^{\varphi_a/2}\,
          d\lambda 
\leq \dfrac {\|u\|_{\varphi_a}}{\Omega_{2n}\delta^{2n}} 
\bigg(\int_{B(z,\delta)}e^{\varphi_a}\, d\lambda \bigg)^{1/2}
\\
&
\leq  a^{-1/2} \Omega_{2n}^{-1/2}\delta^{-n} \|f\|_{\varphi_{a-2}} \cdot
\sup_{|w|\leq \delta} e^{\varphi_a(z+w)/2}.
\nonumber
\end{align}
In our proof of the Siciak-Zakharyuta theorem we need more
sophisticated  methods for deriving uniform estimates from 
$L^2$ estimates.  (See
Hörmander \cite{Hormander:LPDO}, Chapter 15, \cite{Hormander:2018},
p.~328 and  Sigurdsson \cite{Sig:1986}, Lemma~1.3.5.)
These methods are  based on a
formula for  the fundamental solution $E\in \Loneloc(\C^n)\cap
C^{\infty}(\C^n\setminus \{0\})$ of the Laplace operator, 
which for 
$n\geq 2$  is given by
\begin{equation}
  \label{eq:5.12}
  E(z)=-\tfrac 12|z|^{-2n+2}/(\omega_{2n}(n-1))
= -\tfrac 12|z_1\bar z_1+\cdots+z_n\bar z_n|^{-n+1}/(\omega_{2n}(n-1))
\end{equation}
where $\omega_{2n}$ denotes the volume of the unit sphere, and has partial derivatives given by 
\begin{equation}
  \label{eq:5.13}
  \partial_jE(z)=
\tfrac 12|z_1\bar z_1+\cdots+z_n\bar z_n|^{-n}\bar z_j/\omega_{2n}
=\tfrac 12 |z|^{-2n}\bar z_j/\omega_{2n}.
\end{equation}
Recall that for every distribution $v\in {\mathcal E}'(\C^n)$ with
compact support we have
\begin{equation}
  \label{eq:5.14}
  v=E*(\Delta v)=4\sum_{j=1}^n (\partial_j E)*(\bar \partial_j v).
\end{equation}
Assume now that  $u\in C^1(X)$,  $\varphi\colon X\to \overline
\R$  is a measurable function on $X$, $\bar \partial u=f$,
and that $\|u\|_{\varphi_a}\leq a^{-1/2}\|f\|_{\varphi_{a-2}}$ for 
some $a>0$.  Let
     $\chi\in \mathcal{C}^\infty_0(\C^n)$ with $\chi=1$ in some
neighborhood of $0$ and $\chi_z(w)=\chi(z-w)$. By \eqref{eq:5.13} we have for every
$z\in X$ such that $z-K\subset X$ with  $K=\supp \chi$, 
\begin{multline}
  \label{eq:5.15}
u(z)=\chi_z(z)u(z)=
4\sum_{j=1}^n \Big(
\big(\partial_j E\big)*\big(\chi_z\bar\partial_j u\big)(z)
+
\big(\partial_j E\big)*\big(u \bar\partial_j \chi_z\big)(z)
\Big) 
 \\
=
\dfrac 2{\omega_{2n}} \int_{K}
\Big(
|w|^{-2n}\chi(w)\sum_{j=1}^n \bar w_j \bar\partial_j u(z-w)
-|w|^{-2n}\sum_{j=1}^n \bar w_j \bar\partial_j\chi(w)  u(z-w)
\Big)\, d\lambda(w).
\end{multline}
We choose $\chi(w)=\chi_0(|w|^2/\delta^2)$ for $\delta>0$, where
$\chi_0\in \mathcal{C}^\infty(\R)$, $\chi_0(x)=1$ for $x\leq \tfrac 14$, 
$\chi_0$ is decreasing and $\chi_0=0$ for $x\geq 1$.  
Then $\chi(w)=1$ for $|w|\leq \tfrac 12 \delta$, $0\leq \chi\leq 1$, 
and $\chi(w)=0$ for $|w|\geq \delta$.  We have 
$\bar \partial_j \chi(w)=\chi_0'(|w|^2/\delta^2)w_j/\delta^2$ and that
$\chi_0$ can be chosen such that $|\chi_0'|\leq 2$.  We
use this information in the second sum in the integrand in \eqref{eq:5.15} and 
apply the Cauchy-Schwarz inequality to the first sum to get the estimate
\begin{align}
\label{eq:5.16}
|u(z)|&\leq \dfrac 2{\omega_{2n}} \int\limits_{|w|\leq \delta}
|w|^{-2n+1}\, d\lambda(w) \cdot \sup_{|w|\leq \delta}|f(z+w)|
\\
&+  
\dfrac 4{\omega_{2n}\delta^2} \int\limits_{\frac 12 \delta\leq |w|\leq \delta}
|w|^{-2n+2}|u(z-w)| \, d\lambda(w).
\nonumber
\end{align}
Now we apply the Cauchy-Schwarz inequality as in (\ref{eq:5.11})
\begin{align}
\label{eq:5.17}
  |u(z)|&\leq 2\delta \cdot \sup_{|w|\leq \delta}|f(z+w)|
+
\dfrac 4{\omega_{2n}\delta^2} \bigg(\int\limits_{\frac 12 \delta\leq |w|\leq \delta}
|w|^{-4n+4}e^{\varphi_a(z-w)} \,
          d\lambda(w)\bigg)^{1/2}\|u\|_{\varphi_a}
\\
&\leq 2\delta \cdot \sup_{|w|\leq \delta}|f(z+w)|
+a^{-1/2}c_n\delta^{-n} \|f\|_{\varphi_{a-2}} \cdot 
\sup_{|w|\leq \delta} e^{\varphi_a(z+w)/2},
\nonumber
\end{align}
and
\begin{equation}
  \label{eq:5.18}
  c_n=
\dfrac 4{\omega_{2n}} \bigg(\int\limits_{\frac 12 \leq |w|\leq 1}
|w|^{-4n+4} \,  d\lambda(w)\bigg)^{1/2}
=  \begin{cases}
4\big(\log 2/\omega_{2n}\big)^{1/2}, &n=2,\\
4 \big((2^{2n-4}-1)/(\omega_{2n}(2n-4))\big)^{1/2}, &n>2. 
  \end{cases}
\end{equation}
We summarize our calculations in the next result.

\begin{theorem}\label{thm:5.2}  
Let $u\in \mathcal C^1(X)$ denote a solution of the Cauchy-Riemann equations 
given in Theorem \ref{thm:5.1}.  Then for every   $z\in X$  and 
$0<\delta <d(z,\partial X)$ we have 
  \begin{equation}
    \label{eq:5.19}
    |u(z)|
\leq 2\delta \cdot \sup_{|w|\leq \delta}|f(z+w)|
+a^{-1/2}c_n\delta^{-n} \|f\|_{\varphi_{a-2}} 
\sup_{|w|\leq \delta} e^{\varphi_a(z+w)/2},
  \end{equation}
where the constant $c_n$ is given by (\ref{eq:5.18}).
\end{theorem}

In the proof of the Siciak-Zakharyuta theorem we will apply 
Theorem \ref{thm:5.2} and the following variant of 
Lemma 1.3.5 in Sigurdsson \cite{Sig:1986}:

\begin{lemma}\label{lem:5.3}
    Let $v\in {\mathcal C}^1(\C^n)$ 
and assume that there exist positive constants 
$a,C_1,C_2$  and a measurable function $\psi$  
such that for every $z\in \C^n$ we have
    \begin{gather}\label{eq:5.20}
\int_{B(z,1)} |v(w)|^2(1+|e^w|^2)^{-a}e^{-2\psi(w)} d\lambda(w)<C_1
\qquad   \text{  and } \\
\label{eq:5.21}
        |\overline{\partial} v(z)|\leq C_2 e^{\psi(z)}.
    \end{gather}
Then there exists a positive constant $C_3$  such that
    \begin{equation}\label{eq:5.22}
        |v(z)|\leq C_3(1+|e^z|)^a \sup_{w\in \B}e^{\psi(z+w)}, 
\qquad  z\in \C^n. 
    \end{equation}
    \end{lemma}

 \begin{proof} Let $I(z)$ denote the square root of the integral 
in  (\ref{eq:5.20}).
Then, with the same notation as above, (\ref{eq:5.15}) 
holds with $u=v$, $\delta=1$,  and $K={\mathbb B}$, so by (\ref{eq:5.16})
\begin{equation*}
|v(z)|\leq \dfrac 2{\omega_{2n}} \int\limits_{{\mathbb B}}
|w|^{-2n+1}\, d\lambda(w) \sup_{w\in {\mathbb B}}|\bar\partial v(z+w)|
+  
\dfrac 4{\omega_{2n}} \int\limits_{\frac 12 \leq |w|\leq 1}
|w|^{-2n+2}|v(z-w)|\, d\lambda(w).
\end{equation*}
By (\ref{eq:5.21}) and the Cauchy-Schwarz inequality
\begin{equation*}
|v(z)|\leq 2C_2 
\sup_{w\in {\mathbb B}}e^{\psi(z+w)}
+  
\dfrac {4I(z)}{\omega_{2n}} 
\bigg(\int\limits_{\frac 12 \leq |w|\leq 1}
|w|^{-4n+4}(1+|e^{z+w}|^2)^{a}e^{2\psi(z+w)}\, d\lambda(w)\bigg)^{1/2}.
\end{equation*}
Finally, by (\ref{eq:5.18}) and (\ref{eq:5.20}) we get
\begin{equation*}
|v(z)|\leq \big(2C_2+\big(e^aC_1\big)^{1/2}\, c_n\big) 
(1+|e^z|)^{a}\sup_{w\in {\mathbb B}}e^{\psi(z+w)}, \qquad z\in \C^n,
\end{equation*}
and have proved (\ref{eq:5.22}).
    \end{proof}

\section{Proof of the Siciak-Zakharyuta theorem}
\label{sec:06}

Before we go into details of the proof of Theorem  \ref{thm:1.1}
we describe its main ideas. The inequality $V^S_{E,q}\geq \log\Phi^S_{E,q}$ is immeadiate. By  \cite{MagSigSigSno:2023}, Proposition
4.6, $V^S_{E,q}=V^S_{K,q}$ and
$\Phi^S_{E,q}=\Phi^S_{K,q}$ for some compact $K\subseteq E$, so it suffices to prove that $V^S_{K,q}= \log \Phi^S_{K,q}$ for a compact set $K\subseteq \C^n$. Since the admissible weight $q$ is lower-semicontinuous, there is a sequence of continuous functions $q_j\nearrow q$ pointwise on $K$. 
By \cite{MagSigSigSno:2023}, Proposition 4.8 (iii) we have $V^S_{K,q_j}\nearrow V^S_{K,q}$ and $\Phi^S_{K,q_j}\nearrow \Phi^S_{K,q}$, so it suffices to prove that $V^S_{K,q}\leq \log \Phi^S_{K,q}$ on $\C^{*n}$ for continuous weights $q$. Before we apply Theorem \ref{thm:1.2}, we reduce onto the case when $K$ avoids the union of the coordinate hyperplanes $H= \C^n\setminus \C^{*n}$.

\begin{lemma}\label{lem:6.1}
Let $S\subset \R^n_+$ be compact and convex with $0\in S$. Let $q$ be a continuous function on 
a compact subset $K$ of $\C^n$. 
Then there exists a sequence of 
compact subsets $K_j$ of $\C^{*n}$ and continuous functions $q_j$ on $\C^{*n}$ such that $V^{S*}_{K_j,q_j}\leq q_j$ on $K_j$ and
$$\lim_{j\to \infty}V^S_{K_j,q_j}=V^S_{K,q}\qquad\text{ and }\qquad \lim_{j\to \infty}\Phi^S_{K_j,q_j}= \Phi^S_{K,q}.$$ 

\end{lemma}

\begin{proof}
    Let $\varepsilon>0$ and  $\delta_j\searrow 0$, $j=1,2,\dots$. Write $E_j= K+\delta_j\overline{\B}_\infty$, where ${\B}_\infty$ is the open unit ball in the norm $\|\cdot\|_\infty$.  By Tietze's theorem, we can extend $q$ continuously to  $E_1$. Then there is a $j_\varepsilon$ such that for all $j>j_\varepsilon$ we have $\|q(x)-q(y)\|_\infty <\varepsilon$ whenever $x,y\in E_1$ and $\|x-y\|_\infty<2\delta_j$. Let $q_j=V^{S}_{E_j,q}$ and $K_j= (E_j\setminus(H+\frac{1}{2}\delta_j\mathbb{B}_\infty))+\frac{1}{4}\delta_j\overline{\mathbb{B}}_\infty$.

\smallskip
Let $v\in \L^S(\C^n)$ with $v|_{E_j}\leq q$. Then $v\leq V^S_{E_j, q}=q_j$ which implies  $V^S_{E_j,q}\leq V^S_{K_j,q_j}$. Let  $u\in \L^S(\C^n)$ with $u|_{K_j}\leq q_j$. If $z\in K\setminus K_j$ we can arrange the coordinates so $|z_k|< \frac{1}{2}\delta_j$ if $k\leq \ell $ and  $|z_k|\geq \frac{1}{2}\delta_j$ if $k>\ell$. If $w\in \T^\ell\times \{0\}\subset \C^n$ then $z+\delta_j w\in K_j$ and
$$u(z)\leq \max\limits_{w\in \T^\ell\times\{0\}} q_j (z+\delta_j w)\leq \max\limits_{w\in \T^\ell\times\{0\}}  q (z+\delta_j w)\leq q(z)+\varepsilon$$
if $j>j_\varepsilon$. Then $u-\varepsilon\leq q$ on $K$. This implies that $V^S_{K_j,q_j}-\varepsilon\leq V^S_{K,q}$ for all $j>j_\varepsilon$. 

\newpage

 Fix $z_0\in \C^n$. By \cite{MagSigSigSno:2023}, Proposition 4.8 (iii), there is a $j_0\geq j_\varepsilon$ such that for all $j\geq j_0$ we have $V^S_{E_{j}, q}(z_0)+\varepsilon \geq V^S_{K,q}(z_0)$. Then $|V^S_{K,q}(z_0)-V^S_{K_j,q_j}(z_0)|<\varepsilon$ as desired. The argument for $\Phi^S_{K,q}$ is  analoguous, by taking $u=\log|p|^{1/m}$ with $p\in \mathcal{P}^S_m(\C^n)$. 

 By \cite{MagSigSigSno:2023}, Lemma 5.2, Propositions 5.3 and 5.4 (ii), $E_j$ and $K_j$ are locally $\L$-regular,  $q^*_j= V^{S*}_{E_j,q}\leq q$ on $E_j$ and $ q_j=q_j^*\in \mathcal C (\C^{*n})$, therefore $V^{S*}_{K_j,q_j}\leq q_j$ on $K_j$.
\end{proof}

\smallskip
 
By Lemma \ref{lem:6.1} we may assume $K\subset \C^{*n}$ is compact and  $V^{S*}_{K,q}\leq q$ on $K$. 
By Theorem \ref{thm:1.2} there is a convex body $T\subset \R^\ell_+$ and a rational map $F_L\colon \C^{*n}\to \C^{*\ell}$ such that  $V^S_{K,q}={F}_L^* V^T_{K',q'}$ and $\Phi^S_{K,q}= {F}_L^* \Phi^T_{K', q'}$ on $\C^{*n}$, where  $q'={F}_{L*}{q}$ is an admissible weight on the compact set $K'={F}_L(K)\subset \C^{*\ell}$ by Proposition \ref{prop:2.1}. Since $F_L$ is proper, $V^{T*}_{K',q'}\leq q'$ on $K'$, thus by \cite{MagSigSigSno:2023}, Proposition 5.4, $V^{T}_{K',q'}$ is continuous on $\C^{*\ell}$.   

It follows that    
it is sufficient to prove that  $\log\Phi^S_{K,q}\geq V^S_{K,q}$ on $\C^{*n}$
if  $V^S_{K,q}$ is continuous on $\C^{*n}$, $S$ is a convex body and $K \subset \C^{*n}$ is compact. Write $V=V^{S*}_{K,q}$ and note that $V=V^S_{K,q}$ on $\C^{*n}$.

\smallskip
We will prove that for every $z_0\in \C^{*n}$
and $\varepsilon >0$   
there exists a $p\in \OO(\C^n)$ such that for some  $m\in \N^*$,  $0 <a_m < d(mS,\N^n\setminus mS)$ and a constant $C>0$  we have
\begin{gather}
  \label{eq:6.1}
  |p(z)| \leq C(1+|z|)^{a_m} e^{mH_S(z)}, \qquad z\in \C^n,  \\
  \label{eq:6.2}
|p(z)|\leq e^{mV(z)}, \qquad z\in K,  \quad \text{ and } \\
  \label{eq:6.3}
|p(z_0)|=e^{m(V(z_0)-3\varepsilon)}.
\end{gather} 
Then, by (\ref{eq:6.1}) and Proposition 3.5 in
\cite{MagSigSigSno:2023}, 
we have
$p\in {\mathcal P}^S_m(\C^n)$.  Since  $V\leq q$ on $K$, 
(\ref{eq:6.2})
implies that $\|pe^{-mq}\|_K\leq 1$ and (\ref{eq:6.3}) implies
that 
$$\log\Phi^S_{K,q}(z_0) \geq \log\Phi^S_{K,q,m}(z_0) \geq 
\log|p(z_0)|^{1/m}= V(z_0)-3\varepsilon.
$$
Since $z_0$  and $\varepsilon>0$ are arbitrary  we conclude that 
$\log\Phi^S_{K,q}\geq V$ on $\C^{*n}$.

\smallskip
We will apply Theorem \ref{thm:5.1} to show that  
$p$ can be chosen of the form
\begin{equation}
  \label{eq:6.4}
  p(z)=e^{m(V(z_0)-3\varepsilon)}\big(\chi(z)-u(z)\big), \qquad z\in \C^n,
\end{equation}
where $\chi(z)=\chi_0(|z-z_0|^2/\gamma^2)$,  with 
$\chi_0\in \mathcal{C}^\infty(\R)$, $\chi_0(x)=1$ for $x\leq \tfrac 14$, 
$\chi_0$ decreasing, $\chi_0=0$ for $x\geq 1$, 
and $2\gamma$ is chosen smaller than the distance from $z_0$ to the
coordinate hyperplanes.
Then $\bar\partial u=\bar\partial \chi$
and we will choose a certain weight function 
$\varphi\in \PSH(\C^n)$ such that 
\begin{equation}
  \label{eq:6.5}
  \|u\|_{\varphi_{a_m}}\leq a_m^{-1/2} \|\bar \partial \chi\|_{\varphi_{{a_m}-2}}.
\end{equation}
Observe that $\chi(z)=1$ for $|z-z_0|\leq \tfrac 12 \gamma$, 
$0\leq \chi\leq 1$, $\chi(z)=0$ for $|z-z_0|\geq \gamma$,
$\bar \partial_j\chi(z)=\chi_0'(|z-z_0|^2/\gamma^2)(z_j-z_{0,j})/\gamma^2$ 
and that $\chi_0$ can be chosen such that $|\chi_0'|\leq 2$.

\begin{prooftx}{Proof of Theorem \ref{thm:1.1}}
As we have already noted,  we may assume that $S$ is a convex body and $K$ is a compact subset of $\C^{*n}$ and 
it is sufficient to prove that 
for every $z_0\in \C^{*n}$ and $\varepsilon>0$  there
exist $m\in \N^*$ and  $p\in \OO(\C^n)$, such that 
(\ref{eq:6.1}), (\ref{eq:6.2}) and (\ref{eq:6.3}) hold.

\smallskip
Recall that the function  $z\mapsto H_S(z)-\max_{w\in K}H(w)+\min_{w\in K}q(w)$ is
in $\L^S(\C^n)$ and is $\leq q$ on $K$.  This tells us that 
 $V\in \L^S_+(\C^n)$, i.e.~for some $c_V\in \R$ we have
\begin{equation}
  \label{eq:6.6}
  -c_V+H_S(z) \leq V(z) \leq c_V+H_S(z), \qquad z\in \C^n.
\end{equation}
Since $V$ is continuous on $\C^{*n}$ there exists  $\gamma>0$ such that  
\begin{equation}
  \label{eq:6.7}
|V(z+w)-V(z)|<\varepsilon , \qquad z\in  K\cup \{z_0\}, \quad |w|\leq 2\gamma.
\end{equation}
Since $S$ is a convex body, there exists a closed ball 
$\bar B(s_0,2\varrho_0)\subseteq S$.  Then
for every $t\in \bar B(s_0,\varrho_0)$ the closed ball 
$\bar B(t,\varrho_0)$ is contained in $S$.  Its 
supporting  function is 
$\R^n\ni \xi\mapsto \varrho_0|\xi|+\scalar t\xi
\leq \varphi_S(\xi)$, so we have
\begin{equation*}
\varrho_0|\Log\,  z|+\scalar {t}{\Log\,  z}
\leq H_S(z), \qquad t\in B(s_0,\varrho_0), \ \  z\in \C^{*n}.  
\end{equation*}
Observe that if $|z|>|z_0|$ then
$$
\log|z-z_0|\leq \log(|z|+|z_0|)\leq \log|z|+\log 2\leq
\log\|z\|_\infty+\log(2\sqrt n). 
$$
We have $\log\|z\|_\infty\leq |\Log \, z|$ for every $z\in \C^n$, 
so if we choose $c_0\leq -\varrho_0\log(2\sqrt n)$ such that 
$\varrho_0\log|z-z_0|+c_0\leq 0$ for $|z|\leq |z_0|$ then it follows
that
\begin{equation}
\label{eq:6.8}
  \varrho_0\log|z-z_0|+\scalar {t}{\Log\,  z}+c_0 \leq H_S(z),
  \qquad t\in B(s_0,\varrho_0), \ \ z\in \C^{*n}.
\end{equation}
Now we fix $t_0=s_0+\varrho_0 {\mathbf 1}/n\in \bar B(s_0,\varrho_0)$
and observe that for every $m>n/\varrho_0$  
we have $\varphi\in \PSH(\C^n)$, where 
\begin{align}
  \label{eq:6.9}
\varphi(z)= 2(m-n/\varrho_0)(V(z)-\varepsilon)
+(2n/\varrho_0)\big(\varrho_0\log|z-z_0|+ \scalar{t_0}{ \Log\, z} +
   c_0\big).
\end{align}
We choose $\chi$ as described after (\ref{eq:6.4}) and
let $u$ be a solution
of   $\bar \partial u= \bar \partial \chi$  satisfying 
(\ref{eq:6.5}).   Before we make our choice of $m$ let us 
show that (\ref{eq:6.1}) and (\ref{eq:6.3}) hold for every 
$m> n/\varrho_0$.

\smallskip
Since $\varphi(z)= 2n\log|z-z_0|+\psi(z)$,  
where the function $\psi$ is bounded near $z_0$, 
the function $e^{-\varphi}$ is not locally integrable in any neighborhood
of $z_0$ and we conclude that $u(z_0)=0$.  Hence (\ref{eq:6.3}) holds.
By (\ref{eq:6.6}) and (\ref{eq:6.8}) with $t=s_0$ we have for every
$z\in \C^n$  that   
\begin{align*}
  -\varphi(z)+2(m-n/\varrho_0)c_V 
&\geq -2mH_S(z) -2\scalar{\mathbf{1}}{ \Log\, z}\\
&+(2n/\varrho_0)\Big( H_S(z) -\varrho_0\log|z-z_0|-\scalar{s_0}{\Log\, z}-c_0
\Big)
\\
&\geq -2mH_S(z) -2\scalar{\mathbf{1}}{ \Log\, z},
\end{align*}
so  (\ref{eq:6.5}) implies
\begin{equation}
  \label{eq:6.10}
\int_{\C^n} |u(z)|^2(1+|z|^2)^{-{a_m}}  
e^{-2mH_S(z)-2\langle \mathbf{1}, \Log\, z \rangle} d\lambda(z)
\leq e^{2(m-n/\varrho_0)c_V}\|u\|_{\varphi_{a_m}}^2<+\infty.
\end{equation}
The Jacobi determinant of 
$\zeta\mapsto e^\zeta=(e^{\zeta_1},\dots,e^{\zeta_n})$ viewed as a mapping
$\R^{2n}\to \R^{2n}$ is equal to 
$|e^{\zeta_1}\cdots e^{\zeta_n}|^2=e^{2\scalar{{\mathbf 1}}\xi}$.
We define $v\colon \C^n\to \C$ by
$v(\zeta)=u(e^{\zeta})$, write
$\zeta=\xi+i\eta$, for $\xi,\eta\in \R^n$, fix 
$\zeta_0=\xi_0+i\eta_0\in \C^n$
such that  $z_0=e^{\zeta_0}$,
and set 
$$
A=\big(\R\times [\eta_{0,1}-\pi,\eta_{0,1}+\pi]\big)\times\cdots\times 
\big(\R\times [\eta_{0,n}-\pi,\eta_{0,n}+\pi]\big).
$$
By (\ref{eq:6.10}) we have 
\begin{multline*}
\int_{A} |v(\zeta)|^2  (1+|e^{\zeta}|^2)^{-{a_m}}
e^{-2m\varphi_S(\xi)} d\lambda(\zeta)
=
\int_{A} |v(\zeta)|^2 (1+|e^\xi|^2)^{-{a_m}}
e^{-2m\varphi_S(\xi)-2\scalar{\mathbf 1}\xi} e^{2\scalar{\mathbf 1}\xi}
  d\lambda(\zeta) \\
=  \int_{\C^n} |u(z)|^2(1+|z|^2)^{-{a_m}}  
e^{-2mH_S(z)-2\langle \mathbf{1}, \Log\, z \rangle} d\lambda(z)
<+\infty.
\end{multline*}
Since  $\bar\partial u$ has compact support it follows that
there exists a constant $C_2>0$ such that 
\begin{equation*}
|\bar \partial v(\zeta)|\leq C_2 \leq C_2e^{m\varphi_S(\xi)}, \qquad
\zeta=\xi+i\eta\in \C^n.
\end{equation*}
By Lemma \ref{lem:5.3} there exists a constant $C_3>0$ such that 
for every $\zeta=\xi+i\eta\in \C^n$ 
\begin{equation*}
  |v(\zeta)| \leq C_3\big(1+|e^\zeta|\big)^{a_m}
\sup_{w\in {\mathbb B}}e^{m\varphi_S(\xi+\Re\,  w)}
\leq C_4\big(1+|e^\zeta|\big)^{a_m}e^{m\varphi_S(\xi)}, 
\end{equation*}
where $C_4=C_3\sup_{w\in {\mathbb B}}e^{m\varphi_S(\Re \, w)}$.
We change the coordinates back to $z=e^\zeta$, use the fact that
$\Log\, z=\xi$ and conclude that (\ref{eq:6.1}) holds for $p$ as in \eqref{eq:6.4}.

\smallskip
It remains to show that  if $m$ is large enough then 
(\ref{eq:6.2}) holds.   
By Theorem \ref{thm:5.2} we have for every $z\in \C^n$ and $\delta>0$
that
   \begin{multline}
    \label{eq:6.11}
    |p(z)|
\leq e^{m(V(z_0)-3\varepsilon)}\bigg( \chi(z)+
2\delta \cdot \sup_{|w|\leq \delta}|\bar\partial\chi (z+w)|\\
+a_m^{-1/2}c_n\delta^{-n}
\|\bar\partial \chi\|_{\varphi_{{a_m}-2}} 
\sup_{|w|\leq \delta} e^{\varphi_{a_m}(z+w)/2}\bigg),
  \end{multline}
where the constant $c_n$ is given by (\ref{eq:5.18}).
We take $\delta=\gamma$ and $z\in K$. 
If  $z\in \supp \chi$, then (\ref{eq:6.7}) implies that
$V(z_0)-\varepsilon <V(z)$ and we get 
\begin{equation}
\label{eq:6.12}
  e^{m(V(z_0)-3\varepsilon)}\chi(z) \leq e^{-m\varepsilon}e^{mV(z)}.
\end{equation}
This estimate trivially holds for $z\not \in \supp \chi$.
If $z+w\in \supp\bar\partial \chi$ for some $w$ with $|w|\leq \gamma$, then
$|z-z_0|<2\gamma$, so by (\ref{eq:6.7}) 
we have $V(z_0)-\varepsilon <V(z)$
and since 
$|\bar\partial \chi|\leq 2/\gamma$ we get
\begin{equation}
  \label{eq:6.13}
  e^{m(V(z_0)-3\varepsilon)}2\gamma \sup_{|w|\leq
    \gamma}|\bar\partial \chi(z+w)| \leq 4e^{-m\varepsilon}e^{mV(z)}.
\end{equation}
If $z+w\not \in \supp\bar\partial \chi$ then this estimate holds for every $|w|\leq \gamma$.  Next we observe that by (\ref{eq:6.6}),
(\ref{eq:6.7}) and (\ref{eq:6.8}) with $t=t_0$ we have for $z\in K$ and $K_\gamma=\{z\in \C^n \,;\, d(z,K)\leq \gamma\}$
\begin{align}
\label{eq:6.14}
e^{\varphi_{a_m}(z+w)/2}
&\leq \sup_{\zeta\in K_\gamma}(1+|\zeta|^2)^{{a_m}/2}
e^{(m-n/\varrho_0)(V(z+w)-\varepsilon)+(n/\varrho_0)H(z+w)}\\
&\leq \sup_{\zeta\in K_\gamma}(1+|\zeta|)^{a_m}
e^{mV(z)+(n/\varrho_0)(\varepsilon +c_V)}.  
\nonumber
\end{align}
If we combine (\ref{eq:6.11}), 
(\ref{eq:6.12}), (\ref{eq:6.13}) and (\ref{eq:6.14}), 
then we have the estimate
\begin{multline*}
    |p(z)|
\leq e^{-m\varepsilon}\bigg( 5
+a_m^{-1/2}c_n\delta^{-n}
e^{(n/\varrho_0)(\varepsilon+c_V)}\sup_{\zeta\in K_\gamma}(1+|\zeta|)^{a_m} 
e^{m(V(z_0)-2\varepsilon)}\|\bar\partial \chi\|_{\varphi_{{a_m}-2}} 
\bigg)
e^{mV(z)}.
\end{multline*}
If we can prove that 
$e^{m(V(z_0)-2\varepsilon)}\|\bar\partial \chi\|_{\varphi_{{a_m}-2}}$ is 
a bounded function of $m$ then it 
follows that we can choose $m$ sufficiently large for 
(\ref{eq:6.2}) to hold.  For every $\zeta\in \supp \bar \partial
\chi$ we have 
$V(z_0)-2\varepsilon \leq V(\zeta)-\varepsilon \leq V(z_0)$, so 
\begin{multline*}
e^{2m(V(z_0)-2\varepsilon)}\|\bar\partial \chi\|_{\varphi_{{a_m}-2}}^2
\leq \int\limits_{\frac 12\gamma\leq |\zeta-z_0|\leq \gamma}
|\bar\partial \chi(\zeta)|^2(1+|\zeta|^2)^{2}e^{(2n/\varrho_0)V(z_0)}
\\ \times |\zeta-z_0|^{-2n}
\big(|\zeta_1|^{t_{0,1}}\cdots
|\zeta_n|^{t_{0,n}}\big)^{-2n/\varrho_0} e^{-2nc_0/\varrho_0}
\, d\lambda(\zeta).
\end{multline*}
Since $2\gamma$ is smaller than the distance 
from $z_0$ to the coordinate hyperplanes, we have $|\zeta_j|\geq \gamma$ 
for  $|\zeta-z_0|\leq \gamma$, so the integral is convergent
and independent of $m$.
\end{prooftx}

{\small 
\bibliographystyle{siam}
\bibliography{rs_bibref}

\smallskip\noindent
Science Institute,
University of Iceland,
IS-107 Reykjav\'ik,
ICELAND\\
bsm@hi.is, alfheidur@hi.is, ragnar@hi.is
}

\end{document}

%% file: rs_macros.tex
\newcommand{\B}{{\mathbb  B}}
\newcommand{\C}{{\mathbb  C}}

\renewcommand{\L}{{\mathcal L}}

\newcommand{\N}{{\mathbb  N}}
\newcommand{\OO}{{\mathcal O}}

\newcommand{\Q}{{\mathbb  Q}}
\newcommand{\R}{{\mathbb  R}}

\newcommand{\T}{{\mathbb  T}}
\newcommand{\Z}{{\mathbb  Z}}

\newcommand{\scalar}[2]{{\langle#1,#2\rangle}}

\newcommand{\Ker}{{\operatorname{Ker}}}

\newcommand{\Loneloc}{{L_{\text{loc}}^1}}
\newcommand{\Ltwoloc}{{L_{\text{loc}}^2}}
\newcommand{\Log}{{\operatorname{Log}}}

\newcommand{\PSH}{{\operatorname{{\mathcal{PSH}}}}}

\renewcommand{\Re}{{\operatorname{Re}}}

\newcommand{\Span}{{\operatorname{Span}}}

\newcommand{\supp}{{\operatorname{supp}\, }}

\newcommand{\USC}{{\operatorname{{\mathcal{USC}}}}}
\newcommand{\LSC}{{\operatorname{{\mathcal{LSC}}}}}

\def\vfigura#1#2{
\setbox0\vbox{{
\input #1
}}
\setbox1\vbox{\hbox{\box0}\hbox{{\obeylines #2}}}
\dimen0 = -\ht1
\advance\dimen0 by-\dp1
\dimen1 = \wd1
\dimen2 = -\dimen0
\divide\dimen2 by\baselineskip
\count100 = 1
\advance\count100 by\dimen2
\advance\count100 by1
\box1
\hangindent\dimen1
\hangafter=-\count100
\vskip\dimen0
}